\documentclass[a4paper,12pt,dvipdfmx]{article}
\usepackage{amsmath,amssymb,amsfonts,amscd,amsthm}
\usepackage{ascmac} 
\usepackage[mathscr]{eucal}
\usepackage{enumerate}
\usepackage{indentfirst} 
\usepackage{txfonts, textcomp}
\usepackage{graphicx,color}
\newtheorem{thm}{Theorem}[section]
\newtheorem{prop}[thm]{Proposition}

\newtheorem{dfn}[thm]{Definition}
\newtheorem{lem}[thm]{Lemma}

\numberwithin{equation}{section}
\allowdisplaybreaks[4]

\title{Dirac structures on the space of connections }
\author{
Yuji HIROTA\thanks{hirota@azabu-u.ac.jp}\\
Azabu University \and
Tosiaki KORI\thanks{kori@waseda.jp}\\
Waseda University 
}
\date{}
\begin{document}
\maketitle 

\begin{abstract}

We shall give a twisted Dirac structure on the space of irreducible connections on a \( SU(n)\)-bundle over a three-manifold, 
and give a family of twisted Dirac structures on the space  of irreducible connections on the trivial \( SU(n)\)-bundle 
over a four-manifold.
The twist is described by the Cartan 3-form on the space of connections.   It vanishes over the subspace of {\it  flat} connections.  
So the spaces of flat connections are endowed with ( non-twisted ) Dirac structures.    
The Dirac structure on the space of flat connections over the three-manifold is obtained as the boundary restriction of a corresponding Dirac structure over the four-manifold. 
We discuss also the action of the group of gauge transformations over these Dirac structures. 
\end{abstract}

\noindent Subj. Class: Global analysis, Quantum field theory.

\noindent {\bf Keywords}   Dirac structures.   Symplectic structures.    Space of connections.

\section{Introduction}

Let \(X\) be a four-manifold with the boundary three-manifold \(M\).  
Let \(\mathcal{A}_X\) and \(\mathcal{A}_M\) be the spaces of irreducible connections on the principal bundles \(X\times SU(n)\) and \(M\times SU(n)\) respectively.    
In physics language these are spaces of gauge fields.   
We shall investigate twisted Dirac structures on the spaces  \(\mathcal{A}_X\) and \(\mathcal{A}_M\).   
These  twisted Dirac structures are affected by the presence of closed 3-forms.    
Twisted Poisson structures arose from the study of topological sigma models and play an important role in string theory \cite{KSopen97, Ptop01}.    
 P. \v Severa and A. Weinstein, \cite{SW}, investigated  how Poisson geometry on a manifold is affected by the presence of a closed 3-form, and 
they found that the notions of Courant algebroids and Dirac structures provide a framework 
where one can carry out computations in Poisson geometry in the presence of background 3-forms.   

In \cite{K} one of the author proved that there is a pre-symplectic structure on \(\mathcal{A}_X\) that is induced from the canonical symplectic structure 
on the cotangent bundle \(T^{\ast}\mathcal{A}_X\) by a generating function \(CS: \mathcal{A}_X\longrightarrow T^{\ast}\mathcal{A}_X\) which is given by the Chern-Simons form.   
Let \(\omega\) be the boundary reduction of this  pre-symplectic form to  \(\mathcal{A}_M\).    
\(\,\omega\) is no longer pre-symplectic but twisted by the Cartan 3-form on  \(\mathcal{A}_M\).     
 Associated to the 2-form \(\omega\), we have the correspondence: 
\[\omega_A: \,T_A\mathcal{A}_M\,\ni a \,\longmapsto \,\omega_A(a,\cdot\,) \in T^{\ast}_A\mathcal{A}_M\,,\]
  at each \(A\in\mathcal{A}_M\).   
Then we have the subbundle 
\[\mathcal{D}_M\,=\,\left\{\,a\oplus\omega_A(a)\in T\mathcal{A}_X\oplus T^{\ast}\mathcal{A}_X \mid a\in T_A\mathcal{A}_M\,, A\in \mathcal{A}_M\,\right\}.\]
\(\mathcal{D}_M\) gives a twisted Dirac structure of the standard Courant algebroid \(E_0(M)=T\mathcal{A}_M\oplus T^{\ast}\mathcal{A}_M\).    
The 3-form \(\kappa=\tilde{\rm d}\omega\) which describes the twist   
is given by the Cartan 3-form on \(\mathcal{A}_M\):
\[
\kappa_A(a,b,c)=\frac{1}{8\pi^3}\int_M\,{\rm tr}~[abc-bac]\,,\quad a, b, c\in T_A\mathcal{A}_M ,\, A\in \mathcal{A}_M,
\]
where \(abc=a\wedge b\wedge c\) etc..   

On the space of connections \(\mathcal{A}_X\) over a four-manifold \(X\) we consider the correspondence 
\[
\phi_A:\,T_A\mathcal{A}_X\,\ni a\,\longmapsto \,F_Aa+aF_A\,\in T^{\ast}_A\mathcal{A}_X\,,\quad A\in \mathcal{A}_X\,,
\]
where \(F_A\) is the curvature of \(A\).
Then we have the following subbundle of the standard Courant algebroid \(E_0(X)=T\mathcal{A}_X\oplus T^{\ast}\mathcal{A}_X\):
\[\mathcal{D}^{\,\phi}_X\,=\,\left\{\,a\oplus \phi_A(a)\in E_0(X)\mid a\in T_A\mathcal{A}_X\,, A\in \mathcal{A}_X\,\right\}.\]
\(\mathcal{D}_X^{\,\phi}\) gives a twisted Dirac subbundle of \(E_0(X)\).    
The twist in this case is given by the following 3-form \(\kappa\) on \(\mathcal{A}_X\):
\[
\kappa_A(a,b,c)=\kappa_A(\bar{a},\,\bar{b},\,\bar{c}\,),\quad a, b, c\in T_A\mathcal{A}_X,\, A\in \mathcal{A}_X \,,
\]
where \(\bar{a},\bar{b}\) and \(\bar{c} \) indicate the restriction of \(a, b\) and \(c\) respectively to the boundary \(M\), 
and the right-hand side is the 3-form \(\kappa\) on \(\mathcal{A}_M\) described above.  
 
\noindent Moreover if we deal with the correspondence from 
\(T_A\mathcal{A}_X\) to \(T^{\ast}_A\mathcal{A}_X\) given by 
\[
\gamma^{\,t}_A\,(a)\,=\,(F_A+\,t\,A^2)~a\,+\,a~(F_A+\,t\,A^2)\,,\quad \, t \in\mathbb{R}\,,
\]
 then we have the \(\kappa\)-twisted Dirac structure
\[
\mathcal{D}^{\,t}_X\,=\,\left\{\,a\oplus\gamma^{\,t}_A(a)\in E_0(X)\bigm| a\in T_A\mathcal{A}_X\,, A\in \mathcal{A}_X\,\right\}.
\]
 As for the pre-symplectic structure on the space of connections \(\mathcal{A}_X\) we have a family of closed 2-forms on \(\mathcal{A}_X\) given by 
\begin{equation*}
\Omega^{\,t}_A(a,b)=\,\frac{1}{24\pi^3}\int_X{\rm tr}\,
\Bigl[(ab-ba)\,\Bigl\{3F_A\,-(\,t-1)\,A^2\,\Bigr\} \Bigr] -\frac{1}{24\pi^3}\int_{ M}{\rm tr}~\Bigl[(ab-ba)\,A\Bigr]\,.
\end{equation*}
for \(a,\,b\in T_A\mathcal{A}_X\).   
\(\,\Omega^1\) is the pre-symplectic form discussed in \cite{K}.

If we restrict ourselves to the spaces of flat connections we will have (non-twisted) Dirac structures.    In fact we have the following Dirac structures
\begin{eqnarray*}
\mathcal{D}^{\flat}_M\,&=&\,\left\{\,a\oplus \omega_A(a)\in E_0(M)\mid a\in T_A\mathcal{A}^{\flat}_M\,,\,A\in \mathcal{A}^{\flat}_M\,\right\},\\[0.2cm]
\mathcal{D}^{\flat}_X\,&=&\,\left\{\,a\oplus \gamma^1_A(a)\in E_0(X)\mid a\in T_A\mathcal{A}^{\flat}_X\,,\,A\in \mathcal{A}^{\flat}_X\,\right\},
\end{eqnarray*}
where \(\mathcal{A}^{\flat}_M\) and \(\mathcal{A}^{\flat}_X \) are flat connections 
in \(\mathcal{A}_M\) and \(\mathcal{A}_X\)
respectively.
By the boundary restriction from \(X\) to \(M\) we have the correspondence \(r:\,\mathcal{A}_X^{\flat}\longrightarrow\,\mathcal{A}_M^{\flat}\) 
and we find that the image of  \(r\) consists of those flat connections \(A\in \mathcal{A}_M^{\flat}\) with degree \(0\,\): 
\(\deg\,A=\int_M\,{\rm tr}[\,A^3]=0\,\).
   Then \( \mathcal{D}^{\flat,\,\deg 0}_M\,=\{a\oplus \omega_A(a)\in \mathcal{D}^{\flat}_M\mid \deg A=0\,\}\) is a Dirac substructure of \(\,\mathcal{D}^{\flat}_M\,\) and 
there is an isomorphism of Dirac structures between 
\( \mathcal{D}^{\flat}_X\) and \( \mathcal{D}^{\flat,\,\deg 0}_M\) induced by the restriction to the boundary.   


\section{Preliminaries}

\subsection{ Differential calculus on the space of connections} 
Let \(M\) be a compact, connected and oriented \(m\)-dimensional Riemannian manifold  possibly with boundary \(\partial M\).    
Let $G$ be a compact Lie group with Lie algebra $\mathfrak{g}$.    We shall denote by $\varGamma(M, E)$ the space of 
smooth sections of a smooth vector bundle $E\to M$. Especially, 
if $E=TM$  we write $\mathrm{Vect}\,(M)$ for $\varGamma(M, TM)$, and  
if $E=\wedge^kT^*M$ we write $\Omega^k(M)$ for $\varGamma(M, \wedge^kT^*M)$. 
 $\Omega^k(M,E)$ denotes the space $\varGamma(M, \wedge^kT^*M\otimes E)$ of $E$-valued $k$-forms on $M$. 
 
  Let $P$ be a principal \(G\)-bundle on \(M\). A connection 1-form on \(P\) is a \(\mathfrak{g}\)-valued 1-form on \(P\) that is invariant under \(G\), acting by a combination of the action on \(P\) and the adjoint action on \(\mathfrak{g}\).       Let $\mathcal{A}_M$ be the space of irreducible connections on $P$. 
The space $\mathcal{A}_M$ is an affine space modeled on the vector space $\Omega^1(M, \mathrm{ad}\,P)$; the space of 1-forms with values in 
the adjoint bundle ${\rm ad}\,P\,$ over $M$.    For the trivial principal bundle \(P=M\times G\) it is merely the space of $\mathfrak{g}$-valued differential 1-forms on $M$.   
So the tangent space at \(A\in{\mathcal A}_M\) is
 \begin{equation*}
 T_A{\mathcal A}_M=\Omega^1(M,\mathfrak{g})\,.
 \end{equation*}
 The cotangent space at \(A\in\mathcal{A}_M\) is
  \begin{equation*}
 T_A^{\ast}{\cal A}_M=\Omega^{m-1}(M, \mathfrak{g})\,.
 \end{equation*}
The dual pairing of \( T_A{\cal A}_M=\,\Omega^{1}(M, \mathfrak{g})\) and \(T^{\ast}_A{\cal A}_M=\Omega^{m-1}(M, \mathfrak{g})\) is given by 
\[
\langle \alpha,\,a\rangle_A=\int_M\,{\rm tr}\,(\,\alpha\wedge a\,)\,,\quad a\in T_A{\cal A}_M\,,\,\alpha\in T^{\ast}_A{\cal A}_M\,. 
\] 
For a function $H=H(A)$ on $\mathcal{A}_M$ with values in a vector space \(V\),  
the directional derivative $(\partial_{a} H)(A)$ at $A\in \mathcal{A}_M$ to the direction $a\in T_A \mathcal{A}_M$ is defined by  
\[
(\partial_{a} H)\,(A) \coloneqq \lim_{t\to 0}\frac{1}{t}\bigl\{H(A+ta)-H(A)\bigr\}. 
\]
For example, the directional derivative of the identity map ${\rm Id}:A\mapsto A$ on $\mathcal{A}_M$ is 
$(\partial_{\boldsymbol{a}} \,\mathrm{Id})(A) = a$. 
The curvature 2-form of $A\in\mathcal{A}_M$; 
$
F_A = \mathrm{d}A + \frac{1}{2}[A\wedge A] = \mathrm{d}A + A\wedge A
$ 
is viewed as a function $F$ over $\mathcal{A}_M$ with values in $\Omega^2(M,\mathfrak{g})$.  
Since \(
F_{A+ta}-F_A
=t~(\mathrm{d}a + a\wedge A + A\wedge a) + t^2a\wedge a 
\)
we have $(\partial_{a} F)(A) = \mathrm{d}a + a\wedge A + A\wedge a$. 

Let $\boldsymbol{V}$ be a vector field over $\mathcal{A}_M$. The directional derivative $\partial_{a} \boldsymbol{V}$ of 
$\boldsymbol{V}$ to the direction $a$ at $A\in \mathcal{A}_M$
is defined by taking the directional derivative of the coefficients of $\boldsymbol{V}$. 
Namely, if $\boldsymbol{V}$ is written locally as $\boldsymbol{V}=\sum_i\xi_i\partial_i$ with coefficients $\xi_i$, the 
directional derivative $\partial_{\boldsymbol{a}}\boldsymbol{V}$ to the direction $a\in T_A\mathcal{A}_M$ is given by 
\[
(\partial_{a}\boldsymbol{V})(A)\coloneqq \sum_i(\partial_{a}\xi_i)(A)\,\partial_i.
\]
Then the Lie bracket of vector fields $\boldsymbol{V}$ and $\boldsymbol{W}$  on $\mathcal{A}_M$ is given by 
\begin{equation}\label{sec2:eqn_Lie bracket}
 [\boldsymbol{V},\,\boldsymbol{W}](A) =\, 
  (\partial_{\boldsymbol{V}(A)} \boldsymbol{W})(A) - (\partial_{\boldsymbol{W}(A)} \boldsymbol{V})(A)\,,\quad A\in\mathcal{A}_M\,.
\end{equation} 
Let $\theta$ be a $k$-form $(k\geqq 1)$ on the connection space $\mathcal{A}_M$ and $\boldsymbol{X}$ a vector field on $\mathcal{A}_M$. 
The directional derivative $\partial_{\boldsymbol{X}} \theta$ of $\theta$ to the direction of the vector field $\boldsymbol{X}$ is a $k$-form 
that is obtained by the directional derivative of the component functions of $\theta$ to the direction $\boldsymbol{X}(A)$ at each \(A\in \mathcal{A}_M$. 
That is, if $\theta$ is written locally in the form $\theta=\sum_i f_i\boldsymbol{\varepsilon}_i$ with 
component functions $f_i$ and local frames $\{\boldsymbol{\varepsilon}_i\}$ of cotangent bundle $T^*\mathcal{A}_M$, $\partial_{\boldsymbol{X}} \theta$ is given by
\[
(\partial_{\boldsymbol{X}} \theta)_A \coloneqq \sum_i(\partial_{\boldsymbol{X}(A)}f_i)(A)\,\boldsymbol{\varepsilon}_i.\] 
Let  $\langle\theta\mid \boldsymbol{V}\rangle$ denote the evaluation of a 1-form $\theta$ and a vector field $\boldsymbol{V}$.   Then 
it holds that 
\begin{equation}\label{sec2:eqn_evaluation}
\partial_{\boldsymbol{X}}\langle\theta\mid \boldsymbol{V}\rangle = \langle\partial_{\boldsymbol{X}}\theta\mid \boldsymbol{V}\rangle 
+ \langle\theta\mid \partial_{\boldsymbol{X}}\boldsymbol{V}\rangle.
\end{equation}
The exterior derivative of a $k$-form $\theta$ on $\mathcal{A}_M$ 
is the $(k+1)$-form $\,\tilde{\rm d}\theta$ that is given by 
\begin{align*}
 (\tilde{\rm d}\theta)_A\bigl(\boldsymbol{V}_1(A),&\cdots,\boldsymbol{V}_{k+1}(A)\bigr)\\
 &\coloneqq \sum_{i=1}^{k+1}(-1)^{i+1}\bigl(\partial_{\boldsymbol{V}_i} 
    \theta\,(\boldsymbol{V}_1,\cdots,\hat{\boldsymbol{V}_i},\cdots, \boldsymbol{V}_{k+1})\bigr)(A) \\
 &\quad + \sum_{i<j}(-1)^{i+j} \theta_A\,\bigl([\boldsymbol{V}_i,\,\boldsymbol{V}_j](A), \boldsymbol{V}_1(A),\cdots,
  \widehat{\boldsymbol{V}_i(A)},\cdots, \widehat{\boldsymbol{V}_j(A)},\cdots,\boldsymbol{V}_{k+1}(A)\bigr), 
\end{align*}
for any vector fields $\boldsymbol{V}_1,\cdots,\boldsymbol{V}_{k+1}$ on $\mathcal{A}_M$. It can be shown that $\tilde{\rm d}\circ\tilde{\rm d}=0$.    
In particular the exterior derivative of a 1-form $\theta$  becomes  
\begin{equation}\label{sec2:eq derivative}
 (\tilde{\rm d}\theta)(\boldsymbol{V}_1,\boldsymbol{V}_2) 
 = \langle\partial_{\boldsymbol{V}_1}\theta\mid \boldsymbol{V}_2\rangle - \langle\partial_{\boldsymbol{V}_2}\theta\mid \boldsymbol{V}_1\rangle 
\end{equation}
by (\ref{sec2:eqn_Lie bracket}) and (\ref{sec2:eqn_evaluation}).  
 The exterior derivative $\,\tilde{\rm d}\varphi$ of a 2-form $\varphi$ is given by
\begin{equation}\label{sec2:eq_derivative2}
 (\tilde{\rm d}\varphi)_A(\boldsymbol{V}_1,\boldsymbol{V}_2,\boldsymbol{V}_3) 
 = (\partial_{\boldsymbol{V}_1} \varphi)\,(\boldsymbol{V}_2,\, \boldsymbol{V}_3) 
 + (\partial_{\boldsymbol{V}_2} \varphi)\,(\boldsymbol{V}_3,\, \boldsymbol{V}_1) 
 + (\partial_{\boldsymbol{V}_3} \varphi)\,(\boldsymbol{V}_1,\, \boldsymbol{V}_2)\,. 
\end{equation}
The Lie derivative is also defined by the same manner as in the case of finite dimensional smooth manifolds.   
Let $\theta$ be a $k$-form and $\boldsymbol{V}$ be a vector field on $\mathcal{A}_M$, 
the Lie derivative $\mathcal{L}_{\boldsymbol{V}}\theta$ of $\theta$ by $\boldsymbol{V}$ is a $k$-form on $\mathcal{A}_M$ defined by 
\begin{align*}
(\mathcal{L}_{\boldsymbol{V}}\theta)_A\bigl(\boldsymbol{V}_1(A),\cdots,\boldsymbol{V}_k(A)\bigr)
&\coloneqq \bigl(\partial_{\boldsymbol{V}}\theta(\boldsymbol{V}_1,\cdots,\boldsymbol{V}_k)\bigr)(A) \\
&\qquad- \sum_{i}\theta\,\bigl([\boldsymbol{V},\boldsymbol{V}_i](A),\boldsymbol{V}_1(A),\cdots,\widehat{\boldsymbol{V}_i(A)}\cdots,\boldsymbol{V}_k(A)\bigr).
\end{align*}
Especially, for $k=1$, we have 
\begin{equation}\label{sec2:Lie der}
(\mathcal{L}_{\boldsymbol{V}}\theta)(\boldsymbol{W}) 
= \partial_{\boldsymbol{V}}\langle\theta\mid \boldsymbol{W}\rangle - \langle\,\theta\mid [\boldsymbol{V},\boldsymbol{W}\,]\,\rangle.
\end{equation}
For further details of differential calculus on Banach space we refer the readers to \cite{BNnote85, DKgeo90, Sgeo81}.

\subsection{Courant algebroids and  Dirac structures}

The notions of Courant algebroid and Dirac structure are developed  in many forms since T. Courant's work in 1990 \cite{Cdir90}.    
P. \v{S}evera and A. Weinstein 
showed that Courant algebroid and Dirac structure provide a framework to carry out computations in Poisson geometry in the presence of a background 3-form.    
Poisson structures on a manifold \(M\) may be identified with certain Dirac structures in the standard Courant algebroid \(E_0=TM\oplus T^{\ast}M\), 
and a closed 3-form \(\phi\) on \(M\) may be used to modify the bracket on \(E_0\), yielding a new Courant algebroid \(E_{\phi}\).    
Here we shall give a explanatory introduction of these subjects after \cite{Rcou, SW}.

\begin{dfn}~~
A Courant algebroid over a manifold \(M\) is a vector bundle \(E\longrightarrow\,M\) equipped with a field of nondegenerate symmetric bilinear forms 
\((\,\cdot\,,\,\cdot\,)\) on the fibers, an \(\,\mathbb{R}\)-bilinear bracket 
\(\,[\,\cdot\,,\,\cdot\,]\,:\,\varGamma(M,E)\,\times\,\varGamma(M,E)\,\longrightarrow\,\varGamma(M,E)\,\) on the space of sections on \(E\,\), 
and a bundle map \(\rho\,:\,E\longrightarrow\,TM\,\); the {\it anchor}, such that the following properties are satisfied{\rm :}
\begin{enumerate}[\quad \rm(1)]
\item for any
 \(\,e_1,\,e_2,\,e_3\,\in \,\varGamma(M,E)\,\), \(\, [e_1,\,[e_2,\,e_3]\,]\,=\,
[\,[ e_1,\,e_2],\,e_3\,]\,+\,[e_2 ,\,[e_1,\,e_3]\,]\,\){\rm ;}
\item
for any
 \(\,e_1,\,e_2\,\in \,\varGamma(M,E),\quad \rho\,([\,e_1,\,e_2\,])\,=\,[\,\rho(e_1),\,\rho(e_2)\,]\){\rm ;}
 \item
 for any
 \(\,e_1,\,e_2\,\in \,\varGamma(M,E)\,\) and \(f\in C^{\infty}(M)\,\), \([e_1,\,fe_2]\,=\,f[e_1,\,e_2]\,+\,(\rho(e_1)f)\,e_2\,\){\rm ;}
 \item
 for any \(e\,,\,h_1,\,h_2\,\in\,\varGamma(M,E)\,\), \(\rho(e)\,(\,h_1,\,h_2\,)\,=\,(\,[e,\,h_1],\,h_2\,)\,+\,(\,h_1,\,[e,\,h_2]\,)\,\){\rm ;}
 \item
 for any \(\,e\in\,\varGamma(M,E)\,\), \([\,e,\,e\,]\,=\,\mathbb{D}(\,e,\,e\,)\,\),
\end{enumerate}
where \(\mathbb{D}\,=\frac{1}{2}~\beta^{-1}\rho^{\ast}\mathrm{d}\) and \(\beta\) is the isomorphism between \(E\) 
and \(E^{\ast}\) given by the bilinear form{\rm :} \((\beta x)(y)=(x,\,y)\).   
That is, \((\mathbb{D}f,\,e\,)=\frac{1}{2}~\rho(e)f\).
\end{dfn}
The assertion 5 says that there is a linear map \(Z:\varGamma(E)\,\ni e \longrightarrow\,Z_e\in\,\mathrm{Vect}\,(E)\) 
such that \(Z_e\) is a lift of \(\rho(e)\in \mathrm{Vect}\,(M)\), and the first four axioms say that the flow of \(Z_e\) preserves the structure of \(E\).   
The bracket \([\,e_1,\,e_2\,]\) is the Lie derivative of \(e_2\) by \(Z_{e_1}\).

\begin{dfn}~~~
A {\it Dirac structure} in \(E\) is a maximal isotropic subbundle \(\,\mathcal{D}\,\) of \(\,E\) whose sections are closed under the bracket, i.e., 
which is preserved by the flow of \(Z_e\) for every \(e\in \varGamma(\mathcal{D})\).
\end{dfn}

The restriction of the bracket and anchor to any Dirac structure \(\mathcal{D}\) forms a Lie algebroid structure on \(\mathcal{D}\).   

On any manifold we have the standard Courant algebroid \(E_0=TM\oplus T^{\ast}M\) with bilinear form 
\(\left(\,X_1\oplus \xi_1\,,\,X_2\oplus \xi_2\,\right)\,=\,\frac{1}{2}\,\Bigl\{\xi_1(X_2)\,+\,\xi_2(X_1)\Bigr\}\,\), the anchor \(\,\rho\,(X\oplus \xi)=X\,\) 
and the bracket 
\begin{equation}\label{standardCourant}
\left[\,X_1\oplus \xi_1\,,\,X_2\oplus \xi_2\,\right]\,=\,[X_1\,,\,X_2\,]\oplus (\mathcal{L}_{X_1}\xi_2\,-\,i_{X_2}{\rm d}\xi_1)\,.
\end{equation}

Now let \(\phi\) be a 3-form on \(M\).   We define a new bracket on \(E_0\) by adding the term \(\phi(X_1,\,X_2\,,\cdot\,)\) to the right-hand side of 
 (\ref{standardCourant}):
 \begin{equation}\label{3formCourant}
\left[\,X_1\oplus \xi_1\,,\,X_2\oplus \xi_2\,\right]_{\phi}\,=\,[X_1\,,\,X_2\,]\oplus (\mathcal{L}_{X_1}\xi_2\,-\,i_{X_2}{\rm d}\xi_1\,-i_{X_1\wedge X_2}\phi)\,.
\end{equation}

A simple computation shows that the new bracket together 
with the original bilinear form and anchor constitute a Courant algebroid structure on 
\(E_0=TM\oplus T^{\ast}M\) if and only if \({\rm d}\phi=0\).    We denote this modified Courant algebroid by \(E_{\phi}\).    A 
maximal isotropic subbundle \(\mathcal{D}\) of \(E_{\phi}\) whose sections are closed under the bracket \(\left[\cdot\,,\,\cdot\right]_{\phi}\) is called 
a \(\phi\)-{\it twisted Dirac structure}.

We endow the following canonical skew symmetric form on the standard Courant algebroid \(E_0=TM\oplus T^{\ast}M\): 
\begin{equation}\label{standard2form}
\Lambda (a\oplus \alpha\,\mid \,b\oplus \beta)\coloneqq \frac{1}{2}\,\bigl\{\,\langle \alpha\mid b\rangle - \langle \beta\mid a\rangle\,\bigr\} ,
\end{equation}
for $a\oplus \alpha, b\oplus\beta\in E_0$.

\section{Dirac structures on the space of connections}

We shall introduce several Dirac structures on the space of connections over the manifolds of dimension 3 and 4.    First we give an explanation about the related pre-symplectic structures.

\subsection{Pre-symplectic structures on the space of connections}

 Let \(X\) be a four-manifold with the boundary three-manifold \(M\).  
Let \(\mathcal{A}_X\) and \(\mathcal{A}_M\) be the spaces of irreducible connections on the principal bundles \(X\times SU(n)\) and \(M\times SU(n)\) respectively.      
The symplectic structure on the space of connections over a Riemann surface was introduced in 1983 
by M. Atiyah and L. Bott in their study of the geometry and topology of moduli spaces of gauge fields~(see \cite{AByang83}).    
In \cite{K}, we introduced a pre-symplectic structure on the space \({\cal A}_X\) of irreducible connections over a four-manifold \(X\).   
That was given by the 2-form: 
\begin{equation}\label{presympform}
\Omega_A(a,b)=\,\frac{1}{8\pi^3}\int_X{\rm tr}\,[(ab-ba)F_A ] -\frac{1}{24\pi^3}\int_{M}{\rm tr}~[(ab-ba)A]
\end{equation}
for \(a,b\in T_A{\cal A}_X\simeq \Omega^1(X, \mathfrak{g})\).    
We abbreviate often the exterior product of differential forms \(a\wedge b\) to \(ab\,\).   
Let \(\theta\) be the the canonical 1-form on the cotangent bundle \(T^{\ast}{\cal A}_X\), 
and let  \(\sigma=\tilde{\rm d}\theta\) be the canonical 2-form.    
A 1-form  \(\varphi\)  on \({\cal A}_X\) gives a tautological section of the cotangent bundle \(T^{\ast}{\cal A}_X\) so that the pullback 
\(\theta^{\,\varphi}\) of \(\theta\) by \(\varphi\) becomes \(\varphi\) itself:  \(\theta^{\varphi}=\varphi\), 
this is the characteristic property of the canonical 1-form \(\theta\).   
The pullback  \(\sigma^{\varphi}\) of the canonical 2-form \(\sigma\)  by \(\varphi\) is a closed 2-form on \({\cal A}_X\).    
In particular if we take the 1-form on \({\cal A}_X\) given by 
\begin{equation}\label{CS}
CS(A)=\,q~\biggl(AF_A+F_AA-\frac12 A^3\biggr)\,,
\end{equation}
where \(q=\frac{1}{24\pi^3}\), 
then we see that the pullback \(\Omega=\sigma^{CS}\) is given by the equation (\ref{presympform})~(the notation \(CS\) comes from the Chern-Simons function).   
 Thus, for a four-manifold \(X\) there is a pre-symplectic form on \({\cal A}_X\) 
that is induced from the canonical symplectic form on the cotangent bundle \(T^{\ast}{\cal A}_X\) 
by the generating function \(CS\,:{\cal A}_X \longrightarrow T^{\ast}{\cal A}_X\).    
   The quantity 
\(\int_M\,{\rm tr}[A^3]\) for a connection \(A\in\mathcal{A}_X\) plays an analogous role of winding number.     In fact, when \(A\) is a pure gauge; \(A=f^{-1}\tilde{d}f\), it is equal to the degree of \(f\).
We shall change the ratio of counting this number and look for the family of pre-symplectic structures  affected by it.

Put 
\[\Theta^{\,t}_A(a)=q\int_X\,{\rm tr}\,\left[\biggl(AF_A+F_AA-\frac{t}{2}A^3\biggr)\,a\right]\,,\quad (\,t \in \mathbb{R}).\]
By the same observation as above the pullback of the canonical 2-form \(\sigma\) on the cotangent space \(T^{\ast}\mathcal{A}_X\) 
by the generating function \(\Theta^{\,t}\) provides a pre-symplectic structure \(\Omega^{\,t}\) on \(\mathcal{A}_X\).  
    It holds that 
\begin{align*}
\Omega^{\,t}_A(a,b)
&= (\tilde{\rm d}\,\Theta^{\,t})_A(a,b)
\,=\,\langle (\,\partial_{{a}}\Theta^{\,t})_A\mid b\,\rangle\,-\, \langle (\partial_{b}\Theta^{\,t})_A\mid a\,\rangle\,\\
&= q\int_X
\,{\rm tr}\,\Bigl[2(ab-ba)\,F_A \,-\,t\,(ab-ba)\,A^2 \\
&\hspace{100pt} - ({\rm d}_Aa\wedge b-\,a\wedge {\rm d}_Ab\,-\,{\rm d}_Ab\wedge a\,+\,b\wedge {\rm d}_Aa)\wedge A \Bigr]\\
&=\,\,q\int_X
\,{\rm tr}\,[3(ab-ba)\,F_A\,-(\,t-1)(ab-ba)\,A^2]\,-\,q\int_M\,{\rm tr}\,[(ab-ba)\,A] \,,
\end{align*}
because 
\begin{align*}
&{\rm tr}\,\Bigl[
-({\rm d}_Aa\wedge b-\,a\wedge {\rm d}_Ab\,-\,{\rm d}_Ab\wedge a\,+\,b\wedge {\rm d}_Aa)\wedge A \Bigr]\\
=&\,
{\rm tr}\,\Bigl[(ab-ba)\,F_A+(ab-ba)A^2\Bigr]\,
 \,-\,{\rm d}\,{\rm tr}\,\Bigl[(ab-ba)A\Bigr]\,.
\end{align*}

\begin{thm}~~~
We have a family of pre-symplectic structures on the space $\mathcal{A}_X\,$ parametrized by \(t\in \mathbb{R}\){\rm :}
\begin{equation*}
\,\Omega^{\,t}_A(a,b)\,=
\,\Omega_A(a,b)\,-\,(t-1)\, \gamma^{\,\prime}_A(a,\,b) \,,
\end{equation*}
with \(\gamma^{\,\prime}_A(a,\,b)=\,q\int_X\,{\rm tr}\,[(ab-ba)\,A^2]\).
\end{thm}
The case for $t=1$ was discussed in \cite{K}. 

 Every principal \(G\)-bundle over a three-manifold \(M\) is extended to a 
principal \(G\)-bundle over a four-manifold \(X\) that cobords \(M\), and for a connection \(A\in {\cal A}_M\) 
there is a connection \(\mathbf{A}\in {\cal A}_X\) that extends \(A\).    
So we detect  a pre-symplectic structure on  \({\cal A}_M\) as the boundary restriction of the pre-symplectic structure \(\Omega\) on \({\cal A}_X\).
   The boundary 2-form $\omega$ is given by  
\begin{equation}\label{presympform2}
\omega_A(a,b) \coloneqq \frac{1}{24\pi^3}\int_M \mathrm{tr}\,
[(a b - b a) A], \quad A\in\mathcal{A}_M,\,a, b\in T_A\mathcal{A}_M\,.
\end{equation}
 But it is not a closed form and does not give a presymplectic structure on $\mathcal{A}_M$.     
Instead it gives the following 3-form on $\mathcal{A}_M$ that is called {\it the Cartan 3-form}.    
Let 
 $\kappa$ be the 3-form on $\mathcal{A}_M$ defined  by 
\begin{equation}\label{kappa}
\kappa_A(a,b,c) 
 = \frac{1}{8\pi^3}\int_M\mathrm{tr}\,
 [(a b - b a) \,c\,] 
\end{equation} for any $a,b,c\in T_A\mathcal{A}_M$.    It holds that 
$\kappa_A(a,b,c)=\kappa_A(b,c,a)=\kappa_A(c,a,b)$.   We have then 
\begin{equation}
\tilde{\rm d}\omega=\kappa\,.\end{equation}  
 In fact,  (\ref{sec2:eq_derivative2}) yields 
\begin{align*}
&(\tilde{\rm d}\omega)_A(a,b,c) \\
=&\;  (\partial_a\omega)_A\,(b,\, c) + (\partial_b \omega)_A\,(c,\, a) + (\partial_c \omega)_A\,(a,\, b) \\
=&\; 
\frac{1}{24\pi^3}\,\int_M\mathrm{tr}\,[(b\wedge c-c\wedge b)\wedge a] + 
\frac{1}{24\pi^3}\,\int_M\mathrm{tr}\,[(c\wedge a- a\wedge c)\wedge b] \\
&\hspace{180pt} + \frac{1}{24\pi^3}\,\int_M\mathrm{tr}\,[(a\wedge b- b\wedge a)\wedge c] \\ 
=&\; \frac{1}{8\pi^3}\,\int_M\mathrm{tr}\,[(a\wedge b - b\wedge a)\wedge c]. 
\end{align*}
 The 2-form (\ref{presympform2}) gives a pre-symplectic structure twisted by the Cartan 3-form  
\(\,\kappa\).

In the next section we shall describe twisted Dirac structures that are  associated to the twisted pre-symplectic structures discussed above.

\subsection{ Twisted Dirac structure on the space of connections on a 3-manifold $M$}

In this section $M$ is a compact connected oriented Riemannian 3-manifold and $G=S\!U(n)$ with $n\geq 2$.    We shall write  \(\mathfrak{g}= \mathfrak{su}(n)\).   
Let \(P\) be a \(G\)-principal bundle.   Any principal bundle over a 3-manifold has 
a trivialization so that we may assume \(P=M\times\,G\).     
Let \(\Omega^r_s(M,\,\mathfrak{g})\) be the Sobolev space of \(L^2_s\)-sections of \(\mathfrak{g}\)-valued $r$-forms on \(M\).    
It is the completion of \(\{\,\phi\otimes\,X\,|\,\phi\in \Omega^r_s(M),\,X\in \,\mathfrak{g}\,\}\) 
by the \(L^2_s\)-norms.  $\Omega_{s-1}^1(M,\mathfrak{g})$ is a Hilbert space by the inner product 
\begin{equation}\label{sec3:eqn_inner product}
 (\xi\mid \eta) \coloneqq \int_M\langle\xi,\, \eta \rangle\,\mathrm{vol}_M
\end{equation}
for any $\xi,\eta \in \Omega_{s-1}^1(M,\mathfrak{g})$. Here, ${\rm vol}_M$ stands for a volume form on $M$, and $\langle\cdot, \cdot\rangle\,$ is  
the inner product induced from the Riemannian metric on $M$ and the Killing form on $\mathfrak{g}$.

We write \({\cal A}_M\) the space of irreducible connections over \(P\).   \({\cal A}_M\) is an affine space modeled by the vector space \(\Omega^1_{s-1}(M,\,\mathfrak{g})\).   
So the tangent space at \(A\in{\cal A}_M\) is
 \begin{equation*}
 T_A{\cal A}_M=\Omega^1_{s-1}(M, \mathfrak{g})\,, 
 \end{equation*}
that is, the space of differential 1-forms 
with values in $\mathfrak{su}(n)$ whose component functions are of class $L^2_{s-1}$.    
The dual pairing of a 2-form $\alpha\in\Omega_{s-1}^2(M,\mathfrak{g})$ and a 1-form $a\in \Omega_{s-1}^1(M,\mathfrak{g})$ is given by 
\begin{equation*}\label{sec3:evaluation}
\langle \alpha\mid a \rangle_M \coloneqq q\int_M\mathrm{tr}\,(\alpha \wedge a),\quad q=\frac{1}{24\pi^3},
\end{equation*}
which yields the identification of the cotangent space $T^*_A\mathcal{A}_M$ and  $\Omega_{s-1}^2(M,\mathfrak{g})$. 
\medskip

Let \(E_0(M)=\mathbb{T}\mathcal{A}_M\coloneqq T\mathcal{A}_M \oplus T^*\mathcal{A}_M\,\) be the standard Courant algebroid  with the bilinear form 
\begin{equation}\label{sec3:pairing}
\langle \boldsymbol{a}\oplus \alpha,\,\boldsymbol{b}\oplus\beta \rangle_+\coloneqq \frac{1}{2}\,
\bigl\{\,\langle\alpha\mid \boldsymbol{b}\,\rangle_M + \langle\,\beta\mid \boldsymbol{a}\,\rangle_M\,\bigr\}\,,
\end{equation}
the anchor \(\,\rho(\boldsymbol{a}\oplus\alpha)=\boldsymbol{a}\,\), and the bracket 
$\,\llbracket\cdot,\,\cdot\rrbracket\,$ for sections defined by
\begin{equation}\label{sec3:Dorfman}
\llbracket\boldsymbol{a}\oplus\alpha,\,\boldsymbol{b}\oplus\beta\rrbracket\coloneqq 
[\boldsymbol{a},\boldsymbol{b}]\oplus \bigl(\mathcal{L}_{\boldsymbol{a}}\beta-i_{\boldsymbol{b}}\tilde{\rm d}\alpha 
\bigr). 
\end{equation}

We take  the Cartan 3-form  \(\kappa\) on \(\mathcal{A}_M \) of (\ref{kappa}).    
Since \(\kappa=\tilde{\rm d}\omega\) is a closed 2-form we have the Courant algebroid \(E_{\kappa}(M)\) modified by the bracket
\begin{equation}\label{kappaCbracket}
\llbracket\boldsymbol{a}\oplus\alpha,\,\boldsymbol{b}\oplus\beta\rrbracket_{\kappa}\coloneqq 
[\boldsymbol{a},\boldsymbol{b}]\oplus \bigl(\mathcal{L}_{\boldsymbol{a}}\beta-i_{\boldsymbol{b}}\tilde{\rm d}\alpha \,-\,i_{{\boldsymbol{a}}\wedge{\boldsymbol{b}}}\kappa\,
\bigr). 
\end{equation}
Note that the bilinear form and the anchor of \(E_{\kappa}(M)\) are the same ones as for the standard Courant algebroid \(E_0(M)\).

We define a vector subbundle  $ \mathcal{D}_M$  of $E_{\kappa}(M)$  
 by 
\begin{equation}\label{sec3:dfn of D_M}
 \mathcal{D}_M:=\coprod_{A\in\mathcal{A}_M}
  \Bigl\{\,a\oplus (A\wedge a-a\wedge A)\,\bigm|\,a\in T_A\mathcal{A}_M \,\Bigr\}.
\end{equation}

Let \(\omega:\, T\mathcal{A}_M\,\longrightarrow\,T^{\ast}\mathcal{A}_M\,\) be a bundle homomorphism given by 
\begin{equation}\label{sec3:omega}
\omega_A(a)\coloneqq A\wedge a - a\wedge A\,,\quad a\in T_A\mathcal{A}_M\,,\,A\in \mathcal{A}_M.  
\end{equation}
  Any vector field $\boldsymbol{a}:\mathcal{A}_M\to T\mathcal{A}_M$ gives rise to a section $\tilde{\boldsymbol{a}}$ of $\mathcal{D}_M$ by 
$\,\tilde{\boldsymbol{a}}=\boldsymbol{a}\oplus \omega(\boldsymbol{a})$,  where \(
 \omega(\boldsymbol{a})\) denotes the section of \(T^{\ast}\mathcal{A}_M\): $ \omega(\boldsymbol{a})(A)=\omega_A(\boldsymbol{a}(A))$.   
 
 \begin{lem}\label{isotrop}~~~
 $\langle \tilde{\boldsymbol{a}},\tilde{\boldsymbol{b}}\rangle_+=0\,$ for any section $\,\tilde{\boldsymbol{a}},\tilde{\boldsymbol{b}}\,$ of $\,\mathcal{D}_M$. 
 \end{lem}
 
In fact, for $a=\boldsymbol{a}(A),\,b=\boldsymbol{b}(A)\,$, we have 
\begin{align*}
\langle \boldsymbol{a}\oplus\omega(\boldsymbol{a}),\,\boldsymbol{b}\oplus\omega(\boldsymbol{b})\rangle_+ (A)&= 
\frac{1}{2}\,\bigl\{\,\langle\omega_A(a)\mid b\rangle_M + \langle\omega_A(b)\mid a\rangle_M\,\bigr\}\\
&= \frac{q}{2}\int_M{\rm tr}\,\bigl[(A\wedge a-a\wedge A)\wedge b\bigr] + \frac{1}{2}\int_M{\rm tr}\,\bigl[(A\wedge b-b\wedge A)\wedge a\bigr]\\
&= 0.
\end{align*}

The 3-form \(\kappa\) which describes the twist on \(E_0(M)\) 
is defined by
\[
\kappa_A(a,b,c)=\frac{1}{8\pi^3}\int_M\,{\rm tr}~[abc-bac]\,,\quad a, b, c\in T_A\mathcal{A}_M ,\, A\in \mathcal{A}_M,
\]
where \(abc=a\wedge b\wedge c\) etc..   

\begin{lem}\label{Liebracket}~~~
The space of sections $\varGamma(\mathcal{A}_M,\,\mathcal{D}_M)$ of $\mathcal{D}_M$ is closed under the bracket (\ref{kappaCbracket}): 
\begin{equation*}
\llbracket \varGamma(\mathcal{A}_M,\mathcal{D}_M),\,\varGamma(\mathcal{A}_M,\mathcal{D}_M)\rrbracket_{\kappa}\,\subset \varGamma(\mathcal{A}_M,\mathcal{D}_M)\,.
\end{equation*}
\end{lem}
\begin{proof}~~~
It suffices to show that 
\[
\mathcal{L}_{\boldsymbol{a}}\omega(\boldsymbol{b})-i_{\boldsymbol{b}}\tilde{\rm d}\omega(\boldsymbol{a}) 
-i_{\boldsymbol{a}\wedge\boldsymbol{b}}\kappa = \omega([\boldsymbol{a},\boldsymbol{b}])
\]
for any vector field $\boldsymbol{a},\boldsymbol{b}$ on $\mathcal{A}_M$.   
 Let $\boldsymbol{a},\boldsymbol{b}$ and $\boldsymbol{c}$ be vector fields on $\mathcal{A}_M$ and put $a=\boldsymbol{a}(A),\,b=\boldsymbol{b}(A)$ and $c=\boldsymbol{c}(A)$. 
From (\ref{sec2:Lie der}) and (\ref{sec2:eqn_evaluation})
\begin{align}\label{sec3:computation1}
 &\; (\mathcal{L}_{\boldsymbol{a}}\omega(\boldsymbol{b}))_A(\boldsymbol{c}(A)) 
=\; \mathcal{L}_{\boldsymbol{a}}(A\wedge b-b\wedge A)(c) \notag \\
=&\; \partial_a\langle A\wedge b-b\wedge A\mid c\rangle_M(A) - 
     \langle A\wedge b- b\wedge A\mid [\boldsymbol{a},\boldsymbol{c}](A) \rangle_M \notag\\
=&\; \bigl\langle a\wedge b + A\wedge(\partial_a b)-(\partial_a b)\wedge A - b\wedge a \mid c\bigr\rangle_M 
 +\bigl\langle A\wedge b - b\wedge A\mid \partial_c\,a\bigr\rangle_M .
\end{align}
On the other hand, we see from the formula (\ref{sec2:eq derivative}) that $\,i_{\boldsymbol{b}} \tilde{\rm d}\omega(\boldsymbol{a})(\boldsymbol{c})$ is equal to  
\begin{align}\label{sec3:computation2}
 &\; \tilde{\rm d}(\omega(\boldsymbol{a}))_A(b,\,c) 
=\; \langle \partial_b\omega_A(a)\mid c\rangle_M - \langle \partial_c\omega_A(a)\mid b\rangle_M \notag \\
=&\; \bigl\langle b\wedge a + A\wedge (\partial_b a) 
     - (\partial_b a)\wedge A - a\wedge b\mid c\bigr\rangle_M \notag \\
&\qquad\qquad - \bigl\langle c\wedge a + A\wedge (\partial_c a) 
     - (\partial_c a)\wedge A - a\wedge c\mid b\bigr\rangle_M \notag \\
=&\; \bigl\langle b\wedge a + A\wedge (\partial_b a) - (\partial_b a)\wedge A - a\wedge b\mid c\bigr\rangle_M \notag \\
&\qquad \qquad -\bigl\langle a\wedge b-b\wedge a\mid c\bigr\rangle_M -\bigl\langle b\wedge A-A\wedge b \mid \partial_c a\bigr\rangle_M .
\end{align}
By Eqs.(\ref{sec3:computation1}) and (\ref{sec3:computation2}),
\begin{align*}
  & \mathcal{L}_{\boldsymbol{a}}(\omega(\boldsymbol{b}))_A(c) - \tilde{\rm d}(\omega(\boldsymbol{a}))_A(b,\,c)\notag \\
 =&\; \bigl\langle A\wedge (\partial_a b) -(\partial_b a)\wedge A 
   -A\wedge(\partial_b a) + (\partial_b a)\wedge A\mid c \bigr\rangle_M + 3\langle a\wedge b -b\wedge a\mid c\rangle_M \notag \\
 =&\; \langle A\wedge [a,b] - [a,b]\wedge A\mid c\rangle_M 
  + 3\,\langle a\wedge b - b\wedge a\mid c\rangle_M \notag \\
=&\; \langle \omega([\boldsymbol{a},\boldsymbol{b}])\mid c\rangle_A + (i_{\boldsymbol{a}\wedge\boldsymbol{b}}\kappa)_A(c),
\end{align*}
which proves the lemma. 
\end{proof}

The definition of twisted Dirac structure in (\ref{3formCourant}) together with Lemmas \ref{isotrop} and \ref{Liebracket} yield the following theorem.

\begin{thm}~~~
 $\mathcal{D}_M$ is a $\kappa$-twisted Dirac structure on $\mathcal{A}_M$. 
\end{thm}

\subsection{Dirac structures on the space of connections on a 4-manifold $X$ }

\subsubsection{ The twisted Dirac structure induced from the curvature form }

Let $X$ be a four-manifold with the boundary $\partial X=M$ which may be empty. We denote by $\mathcal{A}_X$ the connection space for the trivial bundle 
$X\times G$ over $X$ with \(G=S\!U(n)\).   We denote \(\mathfrak{g}=\mathfrak{su}(n)\) for  $n\geqq 2$. 
The tangent space $T_A\mathcal{A}_X$ at  $A\in\mathcal{A}_X$ is identified with the space $\Omega_{s-\frac12}^1(X,\,\mathfrak{g})$ 
of $\mathfrak{g}$-valued 1-forms, and the cotangent space $T_A^*\mathcal{A}_X$ is identified with 
$\Omega_{s-\frac12}^3(X,\,\mathfrak{g}\,)$.   The pairing of \(T_A\mathcal{A}_X\) and 
$T_A^*\mathcal{A}_X$ is given by 
\[
\bigl\langle \alpha\mid a\bigr\rangle_X \coloneqq \frac{1}{8\pi^3}\int_X{\rm tr}\,(\alpha\wedge a)\,,
\quad\mbox{ for }\,
\alpha\in T_A^*\mathcal{A}_X\,,\, \,a\in T_A\mathcal{A}_X\, .
\]

Let $E_0(X)\,=\mathbb{T}\mathcal{A}_X$ be the standard Courant algebroid over \(\mathcal{A}_X\).     We define a twist  \(\kappa\) on the Courant algebroid \(E_0(X)\)  by the following 3-form:
\begin{equation}\label{def4kappa}
\kappa_A(a,b,c)=\kappa_A(\bar{a},\,\bar{b},\,\bar{c}\,),\quad \mbox{for \(\,a, b, c\in T_A\mathcal{A}_X\,,\,A\in \mathcal{A}_X \)\,,}
\end{equation}
where \(\bar{a},\bar{b}\) and \(\bar{c} \) indicate the restriction of \(a, b\) and \(c\) respectively to the boundary \(M\), 
and the right-hand side is the 3-form \(\kappa\) on \(\mathcal{A}_M\) described in 3.2.    

\begin{lem}\label{sec4:lem01} ~~ $\kappa$ can be represented in the form
\begin{align*}
&\kappa_A(a,b,c)
=
\langle \,{\rm d}_Aa\wedge b\,+\,b\wedge\,{\rm d}_Aa\,\mid c\rangle_X \,
+\,\langle\, {\rm d}_Ab\wedge c+\,c\wedge {\rm d}_Ab\,\mid\,a\rangle_X\,\,+\,\langle\,{\rm d}_Ac\wedge\,a\,+a\wedge {\rm d}_Ac\mid\,b\rangle_X\,,
\end{align*}
for \(\,a,\,b,\,c\in T_A\mathcal{A}_X\,\).
\end{lem}

\begin{proof}~~~ By the Stokes theorem 
\[
\kappa_A(a,b,c) = \,\frac{1}{8\pi^3}\int_M\,\mathrm{tr}\,[\,(ab-ba)\wedge c]\,= 
\,\frac{1}{8\pi^3}\int_X\,{\rm d}\,\mathrm{tr}\,[\,(ab-ba)\wedge c]\,\,.
\]
Since
\begin{eqnarray}
&&\,{\rm d}\,{\rm tr}\,[(ab-ba)c] \nonumber \\
&=&
{\rm tr}\,[(\,{\rm d}_Aa\wedge b-\,a\wedge {\rm d}_Ab\,-\,{\rm d}_Ab\wedge a\,+\,b\wedge {\rm d}_Aa
\,\,)\wedge c \,+\,(\,a\wedge b-b\wedge a\,)\wedge \,{\rm d}_Ac\,] \nonumber \\
&=&
{\rm tr}\,[({\rm d}_Aa\wedge b+\,b\wedge {\rm d}_Aa)\wedge c]\,
\,-{\rm tr}\,[ (\,a\wedge {\rm d}_Ab\,+\,{\rm d}_Ab\wedge a\,)\wedge c] \nonumber \\
&\,&\qquad  +\,{\rm tr}\,[(\,a\wedge b-b\wedge a\,)\wedge \,{\rm d}_Ac], \label{prlemma3.4}
\end{eqnarray} 
the lemma is proved.
\end{proof}

Let $\phi:T\mathcal{A}_X\to T^*\mathcal{A}_X$ be a bundle homomorphism defined by 
\begin{equation}\label{sec3:eqn_bundle homomorphism}
\phi_A(a)\coloneqq F_A\wedge a+a\wedge F_A\,, \quad \mbox{ for }\,a\in T_A\mathcal{A}_X\,,\,A\in \mathcal{A}_X\,,
\end{equation}
where \(F_A\) is the curvature form of \(A\in\mathcal{A}_X\).
Let $ \mathcal{D}^\phi_X$ be the vector subbundle of $E_0(X)\,$  defined by
\begin{equation}
 \mathcal{D}^\phi_X \coloneqq \coprod_{A\in\mathcal{A}_X}
  \Bigl\{\,a\oplus \phi_A(a)\,\bigm|\,a\in T_A\mathcal{A}_X \,\Bigr\}. 
\end{equation}
We shall show that $ \mathcal{D}^\phi_X$ gives a $\kappa$-twisted Dirac structure on $\mathcal{A}_X\,$.  

 First, for  vector fields $\boldsymbol{a}$ and $\boldsymbol{b}$ on $\mathcal{A}_X$,  we have  
\begin{align*}
\langle \phi(\boldsymbol{a})\mid \boldsymbol{b}\rangle_X(A) &= \frac{1}{8\pi^3}\int_X \mathrm{tr}\,\bigl(\phi_A(a)\wedge b\bigr)
= \frac{1}{8\pi^3}\int_X\mathrm{tr}\,(F_A\wedge a\wedge b+a\wedge F_A\wedge b)\\
&= \frac{1}{8\pi^3}\int_X\mathrm{tr}\,[(a\wedge b- b\wedge a)\wedge F_A]\,,
\end{align*}
where $a=\boldsymbol{a}(A)$ and $b=\boldsymbol{b}(A)$.   Then we have the following:
\begin{prop}\label{sec4:prop01}
It holds that 
 \begin{equation}
 \langle \phi(\boldsymbol{a})\mid \boldsymbol{b}\rangle_X+\langle \phi(\boldsymbol{b})\mid \boldsymbol{a}\rangle_X = 0
 \end{equation}
 for any vector fields 
$\boldsymbol{a},\boldsymbol{b}$ on $\mathcal{A}_X\,$.
\end{prop}
Next, we shall show  that the sections of \( \mathcal{D}^\phi_X \) is closed under the bracket \(\llbracket \cdot\,,\,\cdot\rrbracket_{\kappa}\).

\begin{lem}\label{sec4:lem02}
Let $\boldsymbol{a},\boldsymbol{b}$ and $\boldsymbol{c}$ be vector fields on $\mathcal{A}_X$
and put $a=\boldsymbol{a}(A),b=\boldsymbol{b}(A),\,c=\boldsymbol{c}(A)$ for $A\in \mathcal{A}_X$. 
We have the following formulas{\rm :}
\begin{equation}\label{Lie4}
(\mathcal{L}_{\boldsymbol{a}}\phi(\boldsymbol{b}))_A(c)
=\bigl\langle \phi_A(\partial_{a}\boldsymbol{b})\mid c\bigr\rangle_X 
+ \bigl\langle \mathrm{d}_Aa\wedge b + b\wedge \mathrm{d}_Aa\mid c\bigr\rangle_X + \bigl\langle \phi_A(b)\mid (\partial_{c}\boldsymbol{a})(A)\bigr\rangle_X,
\end{equation}
and
\begin{align}\label{idorf}
\bigl(i_{\boldsymbol{b}}{\tilde{\rm d}}\phi(\boldsymbol{a})\bigr)_A(c) 
&= \bigl\langle \phi_A(\partial_{b}\boldsymbol{a})\mid c\bigr\rangle_X + \bigl\langle \phi_A(b)\mid (\partial_c\boldsymbol{a})(A)\bigr\rangle_X \notag \\[0.2cm]
&\qquad -\,\langle\, {\rm d}_Ab\wedge c\,+ \,c\wedge {\rm d}_Ab\,\mid\,a\rangle_X\,
- \bigl\langle a\wedge {\rm d}_Ac\,+\, {\rm d}_Ac\wedge a  \mid b \bigr\rangle_X. 
\end{align}
\end{lem}
\medskip 

\begin{proof}  
Since  $\partial_aF_A=\mathrm{d}_Aa$ for the curvature $F_A$ of $A$.   We have from (\ref{sec2:eqn_evaluation}) and (\ref{sec2:Lie der}), 
\begin{align*}
\,&(\mathcal{L}_{\boldsymbol{a}}\phi(\boldsymbol{b}))_A(c) \\[0.2cm]
=&\; \mathcal{L}_{\boldsymbol{a}}(F_A\wedge b+b\wedge F_A)(c)\\[0.2cm]
=&\; \partial_{\boldsymbol{a}}\bigl\langle F_A\wedge b+b\wedge F_A\mid c\bigr\rangle_X 
- \bigl\langle F_A\wedge b+b\wedge F_A \mid [\boldsymbol{a},\boldsymbol{c}](A)\bigr\rangle_X \\[0.2cm]
=&\; \bigl\langle {\rm d}_Aa\wedge b + F_A\wedge \partial_ab + \partial_ab\wedge F_A + b\wedge {\rm d}_Aa\mid c\bigr\rangle_X 
+ \bigl\langle F_A\wedge b+b\wedge F_A \mid \partial_ac\bigr\rangle_X \\[0.2cm]
&\;\qquad\qquad -\bigl\langle F_A\wedge b+b\wedge F_A\mid \partial_ac-\partial_ca\bigr\rangle_X\\[0.2cm]
=&\; \bigl\langle F_A\wedge \partial_ab + \partial_ab\wedge F_A\mid c\bigr\rangle_X + \bigl\langle {\rm d}_Aa\wedge b + b\wedge {\rm d}_Aa\mid c\bigr\rangle_X 
+\,\bigl\langle F_A\wedge b+b\wedge F_A\mid \partial_ca\bigr\rangle_X\,, 
\end{align*}
which proves the first formula (\ref{Lie4}).  Next, we have
 \begin{align}\label{sec4:lem_2ndeqn}
\,&\bigl(i_{\boldsymbol{b}}{\tilde{\rm d}}\phi(\boldsymbol{a})\bigr)_A(c) \\[0.2cm]
=&\; ({\tilde{\rm d}}(F_A\wedge a+a\wedge F_A)\bigr)_A(b,c) \notag \\[0.2cm]
=&\; \partial_b \bigl\langle F_A\wedge a+a\wedge F_A\mid c\bigr\rangle_X - \partial_c\bigl\langle F_A\wedge a+a\wedge F_A \mid b\bigr\rangle_X  
-\,
 \langle a\wedge F_A+F_A\wedge a\,\mid [\,b,\,c\,]\rangle_X\,\notag
 \\
=&\; \bigl\langle  F_A\wedge \partial_ba + \partial_ba\wedge F_A + a\wedge {\rm d}_Ab +  {\rm d}_Ab\wedge a\mid c\bigr\rangle_X 
+\,\langle a\wedge F_A+F_A\wedge a\mid \partial_bc\rangle
 \notag \\[0.2cm]
&\quad - \bigl\langle   F_A\wedge \partial_ca + \partial_ca\wedge F_A + a\wedge {\rm d}_Ac + {\rm d}_Ac\wedge a \mid b\bigr\rangle_X
-\,\langle a\wedge F_A+F_A\wedge a\mid \partial_cb\rangle_X \notag \\
&\hspace{100pt} - \langle a\wedge F_A+F_A\wedge a \mid \partial_bc\,-\partial_cb\,\rangle_X
\notag \\[0.2cm]
&=\,\langle\,F_A\wedge \partial_ba +\partial_ba\wedge F_A \mid c\rangle_X\,+\,\langle\,F_A\wedge b+b\wedge F_A \mid \partial_ca\,\rangle_X\,\notag\\[0.2cm]
&\,\hspace{100pt} -\,\langle\, c\wedge {\rm d}_Ab+ {\rm d}_Ab\wedge c \mid a\rangle_X\,
- 
\bigl\langle  a\wedge {\rm d}_Ac\,+\,{\rm d}_Ac\wedge a \mid b \bigr\rangle_X
 \,,\notag
\end{align}
where we used $\,\bigl\langle a\wedge {\rm d}_Ab +  {\rm d}_Ab\wedge a\mid c\bigr\rangle_X 
=\,-\, \langle\, c\wedge {\rm d}_Ab+ {\rm d}_Ab\wedge c \mid a\rangle_X\,$.   
The second formula (\ref{idorf}) is proved.
\end{proof}
\bigskip 

  Lemmas \ref{sec4:lem01} and \ref{sec4:lem02}  yield the following:
\begin{prop}\label{sec4:prop02}
For any vector field $\boldsymbol{a},\boldsymbol{b}$ on $\mathcal{A}_X$, 
\[
\mathcal{L}_{\boldsymbol{a}}\phi(\boldsymbol{b})-i_{\boldsymbol{b}}\tilde{\rm d}\phi(\boldsymbol{a}) 
-i_{\boldsymbol{a}\wedge\boldsymbol{b}}\kappa\, = \phi([\boldsymbol{a},\boldsymbol{b}]).
\]
\end{prop}
\medskip 

Propositions \ref{sec4:prop01} and \ref{sec4:prop02} yield the following:  
\begin{thm}\label{sec4:prop_phi}~~~
\begin{equation*}
 \mathcal{D}^\phi_X \coloneqq \coprod_{A\in\mathcal{A}_X}
  \Bigl\{\,a\oplus \phi_A(a)\,\bigm|\,a\in T_A\mathcal{A}_X \,\Bigr\}
\end{equation*}
is a $\kappa$-twisted Dirac structure.
\end{thm}

\subsubsection{A non-twisted Dirac structure on $\mathcal{A}_X$}

On the subspace of flat connections $\mathcal{A}^{\flat}_X\subset \mathcal{A}_X$ the subspace $\mathcal{D}^\phi_X\subset E_0(X)$ reduces to the subspace 
$T\mathcal{A}^{\flat}_X\oplus \{\boldsymbol{0}\}$.   
That implies no knowledge about the Courant algebroid structure on \(T\mathcal{A}^{\flat}_X\oplus T^{\ast}\mathcal{A}^{\flat}_X\).   
In the sequel we shall introduce a more precise Dirac structure over
 $\mathcal{A}_X$ that relates to the winding number of flat connections. 

We introduce the following bundle homomorphism:
\begin{equation}
\gamma^{\,\prime}_A(a)= A^2\wedge a + a\wedge A^2\,, \qquad  A\in \mathcal{A}_X, 
\end{equation}
where $A^2=A\wedge A$. 

\begin{thm}\label{sec4:prop_gamma'}
The subbundle of $E_0(X)\,$ given by 
\begin{equation*}
 \mathcal{D}^{\,\prime}_X \coloneqq \coprod_{A\in\mathcal{A}_X}
  \Bigl\{\,a\oplus \gamma^{\,\prime}_A(a)\,\bigm|\,a\in T_A\mathcal{A}_X \,\Bigr\}
\end{equation*}
is a Dirac structure.
\end{thm}
For the proof we verify the following two conditions:
\begin{enumerate}[\quad(1)]
\item $\bigl\langle \tilde{\boldsymbol{a}},\,\tilde{\boldsymbol{b}}\bigr\rangle_+=0$, 
\item $\mathcal{L}_{\boldsymbol{a}}\gamma^{\,\prime}(\boldsymbol{b})-i_{\boldsymbol{b}}\tilde{\rm d}\gamma^{\,\prime}(\boldsymbol{a})
=\gamma^{\,\prime}([\boldsymbol{a},\boldsymbol{b}])$
\end{enumerate}
for $\tilde{\boldsymbol{a}}=\boldsymbol{a}\oplus \gamma^{\,\prime}(\boldsymbol{a})$ and $\tilde{\boldsymbol{b}}=\boldsymbol{b}\oplus \gamma^{\,\prime}(\boldsymbol{b})$. 
It is easy to check (1) since 
\[
\langle  \gamma^{\prime}(a)\mid \,b\, \rangle_X\,=\,\frac{1}{8\pi^3}\int_X\,{\rm tr}\,(A^2\wedge a+a\wedge A^2)\wedge b
=\,\frac{1}{8\pi^3}\int_X\,{\rm tr}\,[( a\wedge b-b\wedge a)\wedge A^2]\,.\]
 As for (2), the Lie derivative $\mathcal{L}_{\boldsymbol{a}}\gamma^{\,\prime}(\boldsymbol{b})$ is calculated to be 
\begin{align*}
&\bigl(\mathcal{L}_{\boldsymbol{a}}\gamma^{\,\prime}(\boldsymbol{b})\bigr)_A(\boldsymbol{c}) 
=\; \mathcal{L}_{\boldsymbol{a}}(A^2\wedge b + b\wedge A^2)(c) \notag \\
=&\; \partial_{\boldsymbol{a}}\bigl\langle A^2\wedge b + b\wedge A^2\mid c\bigr\rangle_X - \bigl\langle A^2\wedge b + b\wedge A^2\mid 
[\boldsymbol{a},\boldsymbol{c}](A)\bigr\rangle_X \notag \\
=&\; \bigl\langle (a\wedge A+A\wedge a)\wedge b + b\wedge (a\wedge A+A\wedge a) + A^2\wedge\partial_ab + \partial_ab\wedge A^2\mid c\bigr\rangle_X \notag \\
&\hspace{90pt} +\bigl\langle A^2\wedge b + b\wedge A^2\mid \partial_ac\bigr\rangle_X - \bigl\langle A^2\wedge b + b\wedge A^2\mid 
\partial_ac-\partial_ca\bigr\rangle_X \notag \\
=&\; \bigl\langle (a\wedge A+A\wedge a)\wedge b \mid c\bigr\rangle_X + \bigl\langle b\wedge (a\wedge A+A\wedge a) \mid c\bigr\rangle_X 
 \\
&\hspace{75pt} + \bigl\langle A^2\wedge\partial_ab + \partial_ab\wedge A^2\mid c\bigr\rangle_X + \bigl\langle A^2\wedge b + b\wedge A^2\mid \partial_ca\bigr\rangle_X\, . 
\end{align*} 
On the other hand,  we have by (\ref{sec2:eq derivative}) the following:
\begin{align*}
&\bigl(i_{\boldsymbol{b}}{\tilde{\rm d}}\gamma^{\,\prime}(\boldsymbol{a})\bigr)_A(c)\\
=&\; ({\tilde{\rm d}}(A^2\wedge a+a\wedge A^2)\bigr)_A(b,c) \\
=&\; \bigl\langle \partial_b(A^2\wedge a+a\wedge A^2)\mid c\bigr\rangle_X - \bigl\langle \partial_c(A^2\wedge a+a\wedge A^2)\mid b\bigr\rangle_X \\
=&\; \bigl\langle (b\wedge A + A\wedge b)\wedge a + A^2\wedge \partial_ba + \partial_ba\wedge A^2 + a\wedge (b\wedge A+A\wedge b)\mid c\bigr\rangle_X \\
&\quad -\bigl\langle (c\wedge A + A\wedge c)\wedge a + A^2\wedge \partial_ca + \partial_ca\wedge A^2 + a\wedge (c\wedge A+A\wedge c)\mid b\bigr\rangle_X \\
=&\; \bigl\langle (a\wedge A+A\wedge a)\wedge b \mid c\bigr\rangle_X + \bigl\langle b\wedge (a\wedge A+A\wedge a) \mid c\bigr\rangle_X 
\\
&\hspace{75pt} + \bigl\langle A^2\wedge\partial_ba+\partial_ba\wedge A^2\mid c\bigr\rangle_X + \bigl\langle A^2\wedge b + b\wedge A^2\mid \partial_ca\bigr\rangle_X\, .
\end{align*}
Therefore, 
\begin{align*}
\bigl(\mathcal{L}_{\boldsymbol{a}}\gamma'(\boldsymbol{b})\bigr)_A(c) - \bigl(i_{\boldsymbol{b}}{\tilde{\rm d}}\gamma'(\boldsymbol{a})\bigr)_A(c) 
&= \bigl\langle A^2\wedge(\partial_ab-\partial_ba)+(\partial_ab-\partial_ba)\wedge A^2\mid c\bigr\rangle_X \\
&= \bigl\langle A^2\wedge [\boldsymbol{a},\boldsymbol{b}](A) + [\boldsymbol{a},\boldsymbol{b}](A) \wedge A^2\mid c\bigr\rangle_X \\
&= \bigl\langle \gamma^{\,\prime}_A\bigl([\boldsymbol{a},\boldsymbol{b}](A)\bigr)\mid c \bigr\rangle_X\, , 
\end{align*}
which proves (2). 

\subsubsection{A family of $\kappa$-twisted Dirac structures on $\mathcal{A}_X$ }

Now we discuss the perturbation of the $\kappa$-twisted Dirac structure \(\mathcal{D}^{\phi}_X\) by the (non-twisted) Dirac structure \(\mathcal{D}^{\,\prime}_X\). 
We consider the bundle homomorphism \(\,\gamma^{\,t}:\,T\mathcal{A}_X\longrightarrow T^{\ast}\mathcal{A}_X\) given by 
\begin{equation}
\gamma^{\,t}_A(a)\coloneqq \, (F_A+{t}\,A^2)\wedge a + a\wedge(F_A+{t}\,A^2)\,,\quad t\in\mathbb{R}\,.
\end{equation}
We have  $\gamma^{\,t} = \phi\,+\,{\,t}\,\gamma^{\,\prime}$. 

The following theorem is a consequence of Theorems \ref{sec4:prop_phi} and \ref{sec4:prop_gamma'}
\begin{thm}~~~ For \(t\in\mathbb{R}\), put 
\begin{equation}
 \mathcal{D}^{\,t}_X \coloneqq \coprod_{A\in\mathcal{A}_X}
  \Bigl\{\,a\oplus \gamma^{\,t}_A(a)\,\bigm|\,a\in T_A\mathcal{A}_X \,\Bigr\}\,.
\end{equation}
Then \(\bigl\{ \mathcal{D}^{\,t}_X\bigr\}_{{\,t}\in\mathbb{R}}\) gives a family of  $\kappa$-twisted Dirac structures.
\end{thm}

\subsubsection{Dirac structure over the space of flat connections}

Let $\mathcal{A}^{\flat}_X\,$ be the space of flat connections:
\[\mathcal{A}^{\flat}_X=\left\{\,A\in \mathcal{A}_X\mid F_A=0\,\right\}.\]
We have 
 \[T_A\mathcal{A}^{\flat}_X=\{\,a\in \Omega^1_{s-\frac12}(X\,,\,\mathfrak{g}\,)\mid {\rm d}_Aa=0\,\}.\]
 The subbundle \(T\mathcal{A}^{\flat}_X\) is integrable, that is, if \(\mathbf{a}\) and \(\mathbf{b}\) are flat vector fields; 
\( \mathbf{a}(A)\,,\mathbf{b}(A)\,\in T_A\mathcal{A}^{\flat}_X\) for \(\forall A\in \mathcal{A}^{\flat}_X\,\), 
then \(\, [\mathbf{a},\mathbf{b}]\) is a flat vector field.    In fact, this follows from the formula \({\rm d}_A\,[a,b]=[{\rm d}_Aa,b]+ [a,{\rm d}_Ab]\).

 Let 
\begin{equation}
 \mathcal{D}^{\flat}_X \coloneqq \coprod_{A\in\mathcal{A}_X}
  \Bigl\{\,a\oplus \gamma^1_A(a)\,\bigm|\,a\in T_A\mathcal{A}^{\flat}_X \,\Bigr\}\,,
\end{equation}
with 
\begin{equation}
\gamma^1_A(a)\,=\,\gamma^{\prime}_A(a)\coloneqq \, A^2\wedge a + a\wedge A^2\,,\quad 
A\in \mathcal{A}^{\flat}_X.
\end{equation}

Lemma \ref{sec4:lem01} shows that the Cartan 3-form \(\kappa\) on \(\mathcal{A}^{\flat}_X\) vanishes.    
Therefore we have the following:
\begin{thm}~~~
\(\mathcal{D}^{\flat}_X\) is a Dirac structure over \(\mathcal{A}^{\flat}_X\,\subset\,\mathcal{A}_X\) .
\end{thm}

We look at the boundary restriction of the Dirac manifold \( \mathcal{D}^{\flat}_X\, \) to \(M\). 
    Let \({\cal A}^{\flat}_M\) be the space of flat connections;
\[
{\cal A}^{\flat}_M=\{\,A\in{\cal A}_M\mid F_A=0\,\}.
\]    
 The tangent space of  \( {\cal A}^{\flat}_M\) at  \(A\in {\cal A}^{\flat}_M\)  is given by
\begin{equation*}
T_A{\cal A}^{\flat}_M=\{\,a\in\Omega^1_{s-1}(M,\,\mathfrak{g}\,)\mid {\rm d}_Aa=0\,\}.
\end{equation*}
Any principal \(G\)-bundle \(P\) over a 3-manifold \(M\) is extended to a  principal \(G\)-bundle \({\bf P}\) over \(X\), 
and any connection \(A\) on \(P\) has an extension to a connection \({\bf A}\) on \({\bf P}\). 
Then the boundary restriction map 
\(r:\,\mathcal {A}_X\,\longrightarrow \,\mathcal{A}_M\) 
is well defined and surjective.   
The tangent map of \(r\) at \({\bf A}\in{\cal A}_X\) is also given by the restriction to the boundary:
\[
(\partial r)_{{\bf A}} : T_{\bf A}{\mathcal A}_X=\Omega^1_{s-\frac12}(X,\mathfrak{g})\longrightarrow T_A{\cal A}_M=\Omega^1_{s-1}(M,\mathfrak{g}), \quad A=r({\bf A}). 
\] 
In \cite{K}, it is proved that the space of flat connections \(\mathcal{A}_X^{\flat}\) is mapped by \(r\) onto the space of flat connections over \(M\) 
that are of degree \(0\):
 \begin{equation*}
{\mathcal A}^{\,\flat,\,\deg 0}_M=\left\{A\in{\cal A}^{\flat}_M\Bigm| \int_M\!{\rm tr}~A^3=0\,\right\}.
\end{equation*}
Let \(\,\mathcal{D}^{\flat}_M\,\) and \(
\mathcal{D}^{\flat,\,\deg 0}_M\,\) 
 be the subbundles of   \(E^{\flat}(M)\) that are defined by 
\[
\mathcal{D}^{\flat}_M\,=\,\Bigl\{\,a\oplus\omega_A(a)\bigm| a\in T_A\mathcal{A}^{\flat}_M\,,\, A\in\mathcal{A}^{\flat}_M\,\Bigr\},
\]
and
\[
\mathcal{D}^{\flat,\,\deg 0}_M\,=\,\Bigl\{\,a\oplus \omega_A(a)\,\bigm| a\in T_A\mathcal{A}^{\flat,\,\deg 0}_M,\,A\in\mathcal{A}^{\flat,\,\deg 0}_M\,\Bigr\}\,.
\]

\begin{lem}~~For any \(a\in T_A{\mathcal{A}^{\flat}_X}\) we have 
\begin{equation}
\gamma^1_A(a)\,=\,\omega_{r(A)}(\,\partial r_A(a)\,)\,.
\end{equation}
\end{lem}
\begin{proof}~~~
Let \( A\in \mathcal{\mathcal{A}}^{\flat}_X\) and \(a,\,b \in T_{\bf A}{\mathcal{A}^{\flat}_X}\,\).
   Put \(A^{\prime}=r(A)\in {\mathcal{A}^{\flat}_M}\),  \(\,\bar a=\partial r(a)\in T_{A^{\prime}}{\mathcal{A}^{\flat}_M}\,\) and \(\bar b=\partial r(b)\in T_{A^{\prime}}{\mathcal{A}^{\flat}_M}\,\).   We have 
\[\langle \omega_{A^{\prime}}(\bar a)\mid \bar b\rangle
= q\int_M\,{\rm tr}~[ (\bar a\,\bar b- \bar b\, \bar a)\,A^{\prime}]
=q\int_X\,{\rm d}\,{\rm tr}~[(ab- ba)\,A]\, .\,
\]
Since \(F_A=0\) and \({\rm d}_Aa={\rm d}_Ab=0\) for \(\,a,b\in T_A\mathcal{A}_X^{\flat}\,\), by a similar calculation as in (\ref{prlemma3.4}) we find that  
the last formula is equal to 
\[
q\int_X\,{\rm tr}~[(ab- ba)\,{\rm d}_AA\,]\, =\,q\int_X\,{\rm tr}~[(ab- ba)\,A^2\,]\,=\langle \gamma^1_A(a)\mid b\rangle \,.\]
\end{proof}

From the discussion hitherto we are convinced to have the following consequences.

Let \(\,E^{\flat}(M)=\,\mathbb{T}\mathcal{A}^{\flat}_M\,\) be the standard Courant algebroid over the space of flat connections \(\,\mathcal{A}^{\flat}_M\,\).   

\begin{thm}~~
\begin{enumerate}[\quad \rm (1)]
\item 
\(\mathcal{D}^{\flat}_M\,\) 
is a Dirac structure over \(\mathcal{A}^{\flat}_M\,\). 

\item 
\(\mathcal{D}^{\flat,\,\deg 0}_M\,\) 
is a Dirac structure over \(\mathcal{A}^{\flat,\,\deg 0}_M\,\). 
\item
The boundary restriction map \(r\) implies an isomorphism  
between the Dirac structures \(\mathcal{D}_X^{\flat}\,\) and \(\,\mathcal{D}^{\flat,\,\deg 0}_M\).
\end{enumerate}
\end{thm}

\section{The action of the group of gauge transformations on Dirac spaces}

Let \(G=S\!U(n)\) and \(P\longrightarrow N\) be a \(G\)-principal bundle over a \(n\)-dimensional Riemannian manifold $N$ as in subsection 2.1.    
Let \(\mathrm{Ad}\,P\) be the fiber bundle \(\mathrm{Ad}\,P=P\times_GG\longrightarrow\,N\).   
The group of smooth sections of \(\mathrm{Ad}\,P\) under the fiber-wise multiplication is called {\it the group of gauge transformations}.    
We denote by \({\mathcal G}^{\prime}_N\) the group of \(L^2_{s}\)-gauge transformations, i.e., 
\(\,{\mathcal G}^{\prime}_N=\,\Omega^0_{s}(N, \mathrm{Ad}\,P)\, \).
  \({\mathcal G}^{\prime}_N\) acts on \({\mathcal A}_N\) by 
\[\,
g\cdot A=g^{-1}\mathrm{d}g+g^{-1}Ag=A+g^{-1}\mathrm{d}_Ag\,,\quad g\in {\mathcal G}^{\prime}_N . 
\]
\({\cal G}^{\,\prime}_N\) is a Hilbert Lie\,Group  and its action is a smooth map of Hilbert manifolds.      
The action of \(g\in {\mathcal G}^{\prime}_N\) on the tangent space \(T_A\mathcal{A}_N\) is given by 
\[ {\rm Ad}_{g^{-1}}:\,T_A\mathcal{A}_N\ni\,a\,\longmapsto\, {\rm Ad}_{g^{-1}}a=g^{-1}ag\,\in T_A\mathcal{A}_N\,.\]
 We choose a fixed point \(p_0\in N\) and deal with the group of gauge transformations that are identity at \(p_0\):
\[ {\cal G}_N
=\{g\in {\cal G}^{\,\prime}_N\,\mid \,g(p_0)=1\,\}.
\] 
  \({\cal G}_N\) acts freely on \({\mathcal A}_N\).    
We have 
    \(\,{\rm Lie}\,{\cal G}_N\,=\Omega^0_{s}(N, {\rm ad}\,P)\,\).   
The action of \({\rm Lie}\,{\cal G}_N\) on \({\mathcal A}_N\) is given by the covariant exterior derivative:
\[\,
{\rm d}_A={\rm d}+[A\wedge\, \cdot\,]:\, {\rm Lie}\,\mathcal{G}\,=\Omega^0_{s}(N, {\rm ad}\,P)\ni \xi\,\longrightarrow 
{\rm d}_A\xi\in \Omega^1_{s-1}(N,{\rm ad}\,P)=T_A\mathcal{A}_N\,.\]
 So, the fundamental vector field on \({\mathcal A}_N\) corresponding to  
\(\xi \in {\rm Lie}\,{\cal G}_N\) is given by 
\[
{\rm d}_A\xi\,=\left.\frac{d}{dt}\right|_{t=0}\hspace{-.25em}(\exp \,t\xi)\cdot A.
\]    
\(\ker\,{\rm d}_A\,=\,0\) because \(A\in \mathcal{A}_N\) is irreducible by our assumption.     
The tangent space to the orbit at \(A\in{\cal A}_N\) is 
\begin{equation*}
T_A({\cal G}_N\cdot A)=\{{\rm d}_A\xi\,\mid\,\xi\in \Omega^0_{s}(N, {\rm ad}\,P)\}.
\end{equation*}
 Moreover, \(\{{\rm d}_A\xi \,|\,\xi\in {\rm Lie}\,\mathcal{G}_N\}\) is tangent to the space of flat connections \({\cal A}^{\flat}_N\,\),  because 
\({\rm d}_A({\rm d}_A\xi)=\,[\,F_A,\,\xi\,]\,=0\) on \({\cal A}^{\flat}_N\,\).   
Evidently, the condition \(F_A=0\) is \(\mathcal{G}_N\)-invariant, so 
 \({\cal A}^{\flat}_N\) is also \({\cal G}_N-\)invariant manifold,

Now we restrict ourselves to the case treated previously, that is, \(N=X\) is a four-manifold or \(N=M\) is the boundary three-manifold \(M=\partial\,X\).

\subsection{   \(\mathcal{D}^{\,\phi}_X\) under the action of  \(\mathcal{G}_X\) }

   Let  \(\phi_A:\,T_A\mathcal{A}_X\longrightarrow\,T^{\ast}_A\mathcal{A}_X\) be the bundle homomorphism 
\(\phi_A(a)=F_A\wedge a + a\wedge F_A\,\) in (\ref{sec3:eqn_bundle homomorphism}).    
 Since \(F_{g\cdot A}=g^{-1}F_Ag\), we have 
 \begin{equation}
 \phi_{g\cdot A}=g^{-1}\phi_A\,g\,.
  \end{equation}
 We also see that \(\kappa\) is invariant under \(\mathcal{G}_X\).   
 Therefore we have the following; 
\begin{prop}~~~
 The \(\kappa\)-twisted Dirac manifold \(\mathcal{D}^{\,\phi}_X\) is invariant under the action of \(\mathcal{G}_X\).
 \end{prop}

 \subsection{  \(\mathcal{D}_M\) under the infinitesimal action of  \(\mathcal{G}_M\) }
 
   In \cite{K} we saw that 
 the action of \({\cal G}_M\) on the space of flat connections \({\cal A}^{\flat}_M\) is infinitesimally pre-symplectic, that is, 
the Lie derivative of the 2-form $\omega$ in (\ref{presympform2}) by the fundamental vector field \({\rm d}_A\xi\) vanishes:
 \begin{equation}\label{sec4.2:eqn_presymp}
 \mathcal{L}_{{\rm d}_A\xi }\,\omega=0\,,\quad A\in {\cal A}^{\flat}_M\,, 
 \end{equation}
The counterpart of this assertion for the infinitesimal action of \({\cal G}_M\) on the Dirac structure \(\mathcal{D}^{\flat}_M\) holds as well.   We shall explain it in the following.    

We have the canonical skew symmetric form $\Lambda_0$ on \(E_0(M)=T\mathcal{A}_M \oplus T^*\mathcal{A}_M\,\): 
 \begin{equation*}
\Lambda_0 (a\oplus \alpha\,\mid \,b\oplus \beta)\coloneqq \frac{1}{2}\,\bigl\{\,\langle \alpha\mid b\rangle_M - \langle \beta\mid a\rangle_M\,\bigr\} ,
\end{equation*}  
(see (\ref{standard2form})).
We know that the space of connections $\mathcal{A}_M$ admits the Dirac structure 
 \begin{equation*}
 \mathcal{D}_M =\coprod_{A\in\mathcal{A}_M}
  \Bigl\{\,a\oplus \,\omega_A(a)\,\bigm|\,a\in T_A\mathcal{A}_M \,\Bigr\},\quad \,\omega_A(a)=Aa-aA\,.
\end{equation*} 
whose anchor map is \(\rho:\,E_0(M)\,\longrightarrow T\mathcal{A}_M\,\)~(see (\ref{sec3:dfn of D_M})). 
      
 \begin{lem}\label{lambda0omega}~~~
 Restricted to $\mathcal{D}_M$ the canonical 2-form $\Lambda_0$ is the pull back of the pre-symplectic form \(\omega\) by \(\rho\,\):
$\, \Lambda_0\vert \mathcal{D}_M\,=\,\rho^{\ast}\omega $.
 \end{lem}
\begin{proof}
For any section $\boldsymbol{a}\oplus \omega(\boldsymbol{a}), \boldsymbol{b}\oplus\omega(\boldsymbol{b})$ of $\mathcal{D}_M$,  we have 
\begin{eqnarray*}
\Lambda_0\left ( a\oplus \,\omega_A(a)\,\bigm|\,b\oplus \,\omega_A(b)\,\right)
&=&
\frac{1}{2}\,\bigl\{\,\langle Aa-aA\,\mid b\rangle - \langle Ab-bA\,\mid a\rangle\,\bigr\} \\[0.2cm]
&=&\,q\,\int_M\,{\rm tr}~[ (ab\,- \,ba\,)\,A]\,=\,\omega_A(a\,,\,b\,)\,,
\end{eqnarray*}
which is equal to 
\(\,(\rho^{\ast}\omega)( a\oplus \,\omega_A(a)\,,\,b\oplus \,\omega_A(b)\,)\).
\end{proof}

\(\,{\rm Lie}\,\mathcal{G}_M\) acts on  \(\mathcal{D}_M\subset \,E_0(M)\) by the fundamental vector field 
\[
\mathbf{v}_{\xi}\,=\,{\rm d}_A\xi\oplus \omega_A({\rm d}_A\xi)\,, \quad \xi\in {\rm Lie}\,\mathcal{G}_M\,,\,A\in \mathcal{A}_M\,.
\]
The derivation of the 2-form \(\Lambda_0\) along the orbit of \({\rm Lie}\,\mathcal{G}_M\) is given by the Lie derivation
\( \mathcal{L}_{\mathbf{v}_{\xi}(A)}\,\Lambda_0\,\). 
Then, we have the following result. 
\begin{prop}\label{sec4-2:prop_vanishing}~~~
It holds that $\mathcal{L}_{\mathbf{v}_{\xi}(A)}\,\Lambda_0 = 0$ at each $A\in\mathcal{A}^{\flat}_M\,.$    Hence the action of  \({\cal G}_M\) on the Dirac manifold \(\mathcal{D}^{\flat}_M\) is infinitesimally symplectic. 
\end{prop}

\begin{proof}~~~
Since $\rho(\mathbf{v}_{\xi}(A))\,=\,\rho({\rm d}_A\xi\oplus \omega_A({\rm d}_A\xi))={\rm d}_A\xi$,  Lemma \ref{lambda0omega} and the equation (\ref{sec4.2:eqn_presymp}) imply  
\[
\mathcal{L}_{\mathbf{v}_{\xi}(A)}\,\Lambda_0\, = \mathcal{L}_{{\rm d}_A\xi\oplus \omega_A(d_A\xi)}\,(\,\rho^*\omega)=\rho^{\ast}(\mathcal{L}_{{\rm d}_A\xi}\omega)=0.
\]
\end{proof}


\end{document}